\documentclass[12pt]{amsart}
\usepackage{mathrsfs}
\usepackage{url}
\usepackage{color}
\usepackage[a4paper,asymmetric]{geometry}
\usepackage{latexsym}
\usepackage{amsthm}
\usepackage{amssymb}
\usepackage{bbm}
\usepackage{amsfonts}
\usepackage{amsmath}

\newtheorem{theorem}{Theorem}[section]
\newtheorem{thm}[theorem]{Theorem}

\newtheorem{lem}[theorem]{Lemma}
\newtheorem{proposition}[theorem]{Proposition}

\newtheorem{corollary}[theorem]{Corollary}

\theoremstyle{definition}

\newtheorem{defn}[theorem]{Definition}

\theoremstyle{remark}
\newtheorem{remark}[theorem]{Remark}
\newtheorem{rem}[theorem]{Remark}
\numberwithin{equation}{section}

 \DeclareMathAlphabet{\mathpzc}{OT1}{pzc}{m}{it}

 \newcommand{\LLL}{\mathcal{L}}
\newcommand{\1}{\mathbbm{1}}

 \newcommand{\E}{\mathbb{E}}            
 \newcommand{\T}{\mathbb{T}}
 \newcommand{\e}{\varepsilon}

 \newcommand{\Ll}{\langle}
 \newcommand{\Rr}{\rangle}

 \newcommand{\N}{\mathbb{N}}
 \newcommand{\R}{\mathbb{R}}
 \newcommand{\Z}{\mathbb{Z}}

 \newcommand{\PP}{\mathbb{P}}
 \newcommand{\mcl}{\mathcal}
 
 \newcommand{\Be}{\begin{equation}}
 \newcommand{\Ee}{\end{equation}}
 \newcommand{\Bs}{\begin{split}}
 \newcommand{\Es}{\end{split}}
  \newcommand{\Bes}{\begin{equation*}}
 \newcommand{\Ees}{\end{equation*}}
 \newcommand{\BT}{\begin{thm}}
 \newcommand{\ET}{\end{thm}}
 \newcommand{\Bp}{\begin{proof}}
 \newcommand{\Ep}{\end{proof}}
 \newcommand{\BL}{\begin{lem}}
 \newcommand{\EL}{\end{lem}}
 \newcommand{\BP}{\begin{proposition}}
 \newcommand{\EP}{\end{proposition}}
 \newcommand{\BC}{\begin{corollary}}
 \newcommand{\EC}{\end{corollary}}
 \newcommand{\BR}{\begin{rem}}
 \newcommand{\ER}{\end{rem}}
 \newcommand{\BD}{\begin{defn}}
 \newcommand{\ED}{\end{defn}}
 \newcommand{\BI}{\begin{itemize}}
 \newcommand{\EI}{\end{itemize}}
 \newcommand{\eqn}{equation}
 \newcommand{\tl}{\tilde}

  \newcommand{\dif}{{\rm d}}

\def\PP{\mathbb P}
\def\RR{\mathbb{R}}

\def\Om{{\Omega}}

\def\<{\left<}\def\>{\right>}
\def\({\left(}\def\){\right)}

\begin{document}
\title
[\ \ \ \ Irreducibility of stochastic real Ginzburg-Landau equation]{Irreducibility of stochastic real Ginzburg-Landau equation driven by $\alpha$-stable noises and applications}

\author[R. Wang]{Ran Wang}
\address{School of Mathematical Sciences, University of Science and Technology of China, Hefei, China.} \email{wangran@ustc.edu.cn}

 \author[J. Xiong]{Jie Xiong }
\address{Department of Mathematics, Faculty of Science and Technology,  University of  Macau, Taipa, Macau.}
\email{jiexiong@umac.mo}

\author[L. Xu]{Lihu Xu}
\address{Department of Mathematics, Faculty of Science and Technology,  University of  Macau, Taipa, Macau.}
\email{lihuxu@umac.mo}

\maketitle
\begin{minipage}{140mm}
\begin{center}
{\bf Abstract}
\end{center}
We establish the irreducibility of stochastic real Ginzburg-Landau equation with  $\alpha$-stable noises by a maximal inequality
and solving a control problem.
As applications, we prove that the system converges to its equilibrium measure with exponential rate under a topology stronger than total variation and obeys the moderate deviation principle by constructing some Lyapunov test functions.
\end{minipage}

\vspace{4mm}




\section{Introduction}
Consider the stochastic real Ginzburg-Landau equation driven by $\alpha$-stable noises on torus $\mathbb T:=\mathbb R/\mathbb Z$ as follows:
\begin{\eqn} \label{e:MaiSPDE}
\dif X-\partial_{\xi}^2X\dif t-(X-X^3)\dif t= \dif L_t,
\end{\eqn}
where $X:[0,+\infty)\times \mathbb T\times\Om\rightarrow\mathbb R$ and $L_t$ is an   $\alpha$-stable noise.
It is known (\cite{Xu13}) that Eq. \eqref{e:MaiSPDE} admits a unique mild solution $X$ in the c\`adl\`ag space almost surely.
As $\alpha \in (3/2,2)$, $X$ is a Markov process with a unique invariant measure $\pi$ (see Section 2  below for  details).
By the uniqueness (see \cite{DPZ96}), $\pi$ is ergodic in the sense that
$$
\lim_{T\rightarrow \infty}\frac1T\int_0^T \Psi(X_t)\dif t=\int \Psi \dif \pi\ \ \ \ \mathbb P {\text-a.s.}
$$
for all initial state $x_0$ and all continuous and bounded functions $\Psi$.

\vskip0.3cm
The irreducibility is a fundamental concept in stochastic dynamic system, and plays
a crucial role in the research of ergodic theory.
See, for instance, the classical work \cite{Doob} and the books for stochastic infinite dimensional systems \cite{DPZ96, PeZa07}.

It is well known that one usually solves a control problem
to prove the irreducibility for stochastic partial differential equations (SPDEs) driven by Wiener noises, see \cite{Da, DPZ96}.
For SPDEs driven by $\alpha$-stable noises, when the system is linear or Lipschitz,
Priola and Zabczyk proved the irreducibility in the same line (\cite{PZ11}). However, due to the
discontinuity of trajectories and the lack of second moment, the control problem in \cite{PZ11} is much harder
than those in the Wiener noises case. To our knowledge, there seem no other literatures about the irreducibility of stochastic systems
with $\alpha$-stable noises.

In this paper, we prove that the system \eqref{e:MaiSPDE} is irreducible by following the spirit in
\cite{Da} and \cite{PZ11}. Due to the non-Lipschitz nonlinearity,
the control problem in our setting is much harder and a maximal inequality is needed.
\vskip0.3cm

The ergodicity of the system \eqref{e:MaiSPDE} has been proved  in \cite{Xu13} in the sense that $X$ converges to a unique invariant measure
under the weak topology,  but the convergence speed is not addressed.
In this paper, thanks to the irreducibility and the strong Feller property (established in \cite{Xu13}),
we prove that the system \eqref{e:MaiSPDE} converges to the invariant measure exponentially fast under a
topology stronger than total variation by constructing a Lyapunov test function.

Another application of our irreducibility result is to establish moderate deviation principle (MDP) of \eqref{e:MaiSPDE}.
Thanks to \cite{Wu01}, the MDP is obtained by verifying the same Lyapunov condition as above.

\vskip0.3cm
Finally we recall some literatures on the study of invariant measures and the long time behavior of stochastic systems driven by $\alpha$-stable type noises. \cite{PSXZ11, PXZ10} studied the exponential mixing for a family of semi-linear SPDEs with Lipschitz nonlinearity, while \cite{DXZ09} obtained the existence of invariant measures for 2D stochastic Navier-Stokes equations forced by $\alpha$-stable noises with $\alpha\in (1,2)$. \cite{Xu14} proved the exponential mixing for a family of 2D SDEs forced by degenerate $\alpha$-stable noises.
For the long term behaviour about stochastic system drive by L\'evy noises,
we refer to \cite{ Do08, DXi11, DXZ09, EH01, FuXi09, Mas} and the literatures therein.
\vskip0.3cm

The paper is organized as follows. In Section 2, we first give a brief review of some known results about the existence and uniqueness of solutions and invariant probability measures for  stochastic  Ginzburg-Landau equations. We will also present the main theorems in this section. In Section 3,  we prove that the system $X$ is  irreducible.   In the last section,  we first recall  some results about moderate deviations and exponential convergence for general strong Feller Markov processes, and then we prove moderate deviations and exponential convergence for $X$ by constructing appropriate Lyapunov test functions.

\section{Stochastic real Ginzburg-Landau equations}

Let $\T= \R/\Z$ be equipped with the usual Riemannian metric, and let $d \xi$
denote the Lebesgue measure on $\T$. For any $p\ge1$, let
$$
L^p(\T;\R):=\left\{x: \T\rightarrow\R; \|x\|_{L^p}:=\left(\int_\T |x(\xi)|^4 \dif\xi\right)^{\frac14}<\infty\right\}.
 $$
Then
$$H:=\bigg\{x\in L^2(\T; \R); \int_\T x(\xi) \dif\xi =0\bigg\}$$
is a separable real Hilbert space with inner product
$$\Ll x,y \Rr_H:=\int_\T x(\xi)y(\xi) \dif\xi,\ \ \ \ \ \forall \ x, y \in H.$$
For any $x\in H$, let
$$
\|x\|_H:=\|x\|_{L^2}=\left(\langle x,x\rangle_H\right)^{\frac12}.
$$
Let $\Z_*:=\Z \setminus \{0\}$. It is well known that
$$\left\{e_k; e_k=e^{i2 \pi k\xi}, \ k \in \Z_*\right\} $$
is an orthonormal basis of $H$. For each $x \in H$,
it can be represented by  Fourier series
$$x=\sum_{k \in \Z_*} x_k e_k  \ \ \ \ {\rm with} \ \ \ x_k \in \mathbb C, \ x_{-k}=\overline{x_k}.$$

Let $\Delta$ be the Laplace operator on $H$. It is well known that
$D(\Delta)=H^{2,2}(\T) \cap H$. In our setting, $\Delta$ can be determined by the following
relations: for all $k \in \Z_*$,
$$\Delta e_k=-\gamma_k e_k\ \ \ \ {\rm with} \ \ \gamma_k=4 \pi^2 |k|^2,$$
with
$$H^{2,2}(\T) \cap H=\left\{x \in H; \ x=\sum_{k \in \Z_*} x_k e_k, \ \sum_{k \in \Z_*} |\gamma_k|^{2} |x_k|^2<\infty\right\}.$$
Denote
$$A=-\Delta, \ \ \ \ D(A)=H^{2,2}(\T) \cap H.$$
Define the operator $A^{\sigma}$ with $\sigma \ge 0$ by
$$A^\sigma x=\sum_{k \in \Z_*} \gamma_k^{\sigma} x_ke_k, \ \ \ \ \ \ x \in D(A^\sigma),$$
where $\{x_k\}_{k \in \Z_*}$ are the Fourier coefficients of $x$, and
$$D(A^\sigma):=\left\{x \in H: \ x=\sum_{k \in \Z_*} x_k e_k, \sum_{k \in \Z_*} |\gamma_k|^{2 \sigma} |x_k|^2<\infty\right\}.$$
Given $x \in D(A^\sigma)$, its norm is
$$\|A^\sigma x\|_H:=\left(\sum_{k \in \Z_*} |\gamma_k|^{2 \sigma} |x_k|^2\right)^{1/2}.$$
Moreover, let
\begin{align}
V:=D(A^{1/2})\ \ \text{and }  \|x\|_V:=\|A^{1/2}x\|_H.
\end{align}
Notice that $V$ is densely and compactly embedded in $H$. \\

We shall study  $1$D stochastic Ginzburg-Landau equation on $\T$ as the following
\begin{equation} \label{e:XEqn}
\begin{cases}
\dif X_t + A X_t\dif t= N(X_t) \dif t + \dif L_t, \\
X_0=x_0,
\end{cases}
\end{equation}
where
\begin{itemize}
\item[(i)] the nonlinear term $N$ is defined by
\begin{equation*} \label{e:NonlinearB}
N(u)= u-u^3 , \ \ \ \ \ u \in H.
\end{equation*}
\item[(ii)] $L_t=\sum_{k \in \Z_*} \beta_k l_k(t) e_k$ is an $\alpha$-stable process on $H$ with
$\{l_k(t)\}_{k\in \Z_*}$ being i.i.d. 1-dimensional symmetric $\alpha$-stable process sequence with $\alpha>1$, see \cite{sato}. Moreover, we assume that there exist some $C_1, C_2>0$ so that $C_1 \gamma_k^{-\beta} \le |\beta_k| \le C_2 \gamma_k^{-\beta}$ with $\beta>\frac 12+\frac 1{2\alpha}$.
\end{itemize}

\vskip0.3cm

Let $C>0$ be a constant and let $C_p>0$ be a constant depending on the parameter $p$.
We shall often use the following inequalities (\cite{Xu13}):
\Be \label{e:eAEst}
\|A^{\sigma} e^{-At}\|_H \le C_\sigma t^{-\sigma}, \ \ \ \ \ \forall \ \sigma>0 \ \ \ \forall \ t>0;
\Ee
\Be \label{e:NVEst}
\|N(x)\|_V \leq C (\|x\|_V+\|x\|^3_V), \ \ \ \ \forall \ x \in V;
\Ee
\Be\label{e: NAH}
\|AN(x)\|_H\le C(1+\|x\|_V^2)(1+\|Ax\|_H^2);
\Ee
\Be \label{e:L4}
\|x\|^4_{L^4}\le \|x\|_V^2\|x\|_H^2, \ \ \ \ \  \forall  \ x\in V.
\Ee

\vskip0.3cm
Here we consider a general $E$-valued c\`adl\`ag  Markov process,
$$
(\Omega, \{\mathcal F^0_t\}_{t\ge0}, \mathcal F, \{X_t^x\}_{t\ge0, x\in E}, (\mathbb P_x)_{x\in E})
$$
whose transition probability is denoted by $\{P_t(x, dy)\}_{t\ge0}$, where $\Omega:=D( [0,+\infty); E)$ is the space of the c\`adl\`ag functions from $[0,+\infty)$ to $E$ equipped with the  Skorokhod topology, $\mathcal F^0_t=\sigma\{X_s, 0\le s\le  t\}$ is the natural filtration  .

For all $f\in b\mathcal B(E)$ (the space of all bounded measurable functions), define
$$
P_tf(x)=\int_E P_t(x, \dif y)f(y)  \ \ \ \text{for all } t\ge0, x\in E.
$$
For any $t>0$, $P_t$ is said to be {\it strong Feller} if $P_t\varphi\in C_b(E)$ for any $\varphi\in b\mathcal B(E)$; $P_t$ is {\it irreducible} in $E$ if $P_t1_O(x)>0$ for any $x\in E$ and any non-empty open subset $O$ of $E$.
\vskip0.3cm

\BD We say that a predictable $H$-valued stochastic process $X=(X_t^x)$ is a mild solution to Eq. \eqref{e:XEqn} if, for any $t\ge0, x\in H$, it holds ($\mathbb P$-a.s.):
\begin{equation}\label{e: mild solution}
X^x_t(\omega)=e^{-At} x+\int_0^t e^{-A(t-s)} N(X_s^x(\omega))\dif s+\int_0^t e^{-A(t-s)} \dif L_s(\omega).
\end{equation}
\ED

The following existence and uniqueness results for the solutions and the invariant measure can be found in \cite{Xu13}.
\begin{thm}[\cite{Xu13}] \label{Thm Xu 13} The following statements hold:
\begin{enumerate}
\item  If $\alpha \in (1,2)$ and $\beta>\frac 12+\frac{1}{2\alpha}$, for every $x \in H$ and $\omega \in \Omega$ a.s.,
Eq. \eqref{e:XEqn} admits a unique mild solution $X^x_{\cdot}(\omega) \in D([0,\infty);H) \cap D((0,\infty);V)$.

\item  $X=(X^x_t)_{t\ge0,x\in H}$ is a Markov process. If  $\alpha \in (3/2,2)$ and $\frac 12+\frac{1}{2\alpha}<\beta<\frac 32-\frac{1}{\alpha}$ are further assumed,  the transition probability $P_t$ of $X$ is strong Feller in $H$ for any $t>0$.
\item
 If $\alpha \in (3/2,2)$ and $\frac 12+\frac{1}{2\alpha}<\beta<\frac 32-\frac{1}{\alpha}$, $X$ admits a unique invariant measure, and the invariant measure is supported on $V$.
\end{enumerate}
\end{thm}

Our first main result is the following theorem about the irreducibility.

\begin{thm}\label{thm irred}
Assume that $\alpha \in (1,2)$ and $\beta>\frac 12+\frac{1}{2\alpha}$.
For any initial value $x\in H$, the  Markov process $X=\{X_t^x\}_{t\ge0,x\in H}$ to the equation \eqref{e:XEqn} is irreducible in $H$.
\end{thm}

\begin{remark}\label{rmk} By the well-known Doob's Theorem (see \cite{DPZ96}), the strong Feller property and the irreducibility imply that  $X$ admits at most one unique invariant probability measure. This   gives another proof to the uniqueness of invariant measure.
 \end{remark}
\vskip0.3cm

As an application of our irreducibility result (together with strong Feller property), we have the following exponential ergodicity under a topology stronger than total variation.
Recall that in  \cite{Xu13} by ergodicity we mean that  $X$ has a unique invariant measure
under the weak topology. Theorem \ref{thm main 1} below gives not only an ergodic theorem in stronger sense but also exponential convergence speed.

\begin{thm}\label{thm main 1}
Assume that $\alpha \in (3/2,2)$ and $\frac 12+\frac{1}{2\alpha}<\beta<\frac 32-\frac{1}{\alpha}$. Let $\pi$ be the unique invariant probability  measure of $X$. Then there exist some positive constants $M>1, \rho\in(0,1), \theta>0$ satisfying that  $\int \Psi\dif \pi<+\infty$, where $\Psi(x):= (M+\|x\|_H^2)^{1/2}$,
  and $\pi$ is exponentially ergodic in the sense that
\begin{equation}\label{thm exp}
\sup_{|f|\le \Psi}\left|P_tf(x)-\int f\dif \pi \right|\le \theta \Psi(x)\cdot \rho^t\ \ \  \forall x\in H,  t\ge0.
\end{equation}
 \end{thm}

\begin{remark}
  Let $(B_{\Psi}, \|\cdot\|_{\Psi})$ be the Banach space of all real measurable functions $f$ on $H$ such that
$$
\|f\|_{\Psi}:=\sup_{x\in H}\frac{|f(x)|}{\Psi(x)}<+\infty.
$$
The exponential convergence \eqref{thm exp} means that
$$
\|(P_t-\pi)(f)\|_{\Psi}\le \theta\|f\|_{\Psi}\cdot\rho^t,
$$
i.e., $P_t$ has  a spectral gap near its largest eigenvalue $1$ in $B_{\Psi}$.
\end{remark}
\vskip0.3cm

Let $\mathcal M_b(H)$ be the space of signed $\sigma$-additive measures of bounded variation on $H$ equipped with the Borel $\sigma$-field $\mathcal B (H)$.
On  $\mathcal M_b(H)$, we   consider the topology $\sigma(\mathcal M_b(H), b\mathcal B(H))$,  the so called $\tau$-topology  of convergence against  all bounded Borel functions  which is   stronger than the usual weak convergence topology $\sigma(\mathcal M_b(H), C_b(H))$, see \cite[Section 6.2]{DZ}.

Let
$$
\mathfrak{L}_t(A):=\frac1t\int_0^t\delta_{X_s}(A)\dif s \ \ \  \ \text{ for any measurable set } A,
$$
where $\delta_a$ is the Dirac measure at $a$.  According  to Corollary 2.5 in \cite{Wu01}, the system $X$ has the following exponential ergodicity.
\begin{corollary} Under the conditions of Theorem \ref{thm main 1}, the following results hold:
\begin{itemize}
  \item[(a)] $\mathfrak{L}_t$ converges to $\pi$ with an exponential rate w.r.t. the $\tau$-topology. More precisely for any neighborhood $\mathcal N(\pi)$ of $\pi$ in $(\mathcal M_b(H),\tau)$,
      $$
      \sup_{K\subset\subset H}\limsup_{t\rightarrow+\infty}\frac1t\log\sup_{x\in K}\mathbb P_x(\mathfrak{L}_t \notin \mathcal N(\pi))<0.
        $$
        Here $K\subset\subset H$ means that $K$ is a compact set in $H$.
  \item[(b)] The process $X$ is exponentially recurrent in the sense below: for any compact $K$ in $H$ with  $\pi(K)>0$, there exists some $\lambda_0>0$ such that for any compact $K'$ in $H$,
      $$
      \sup_{x\in K'}\mathbb E_x \exp(\lambda_0 \tau_K(T))<+\infty, \
      $$
      where $\tau_K(T)=\inf\{t\ge T; X_t\in K \}$ for any $T>0$.

   \end{itemize}
\end{corollary}

 \vskip0.3cm

Another application of our irreducibility result (together with strong Feller property) is to establish MDP for
the system \eqref{e:MaiSPDE}. To this end, let us first briefly recall MDP as follows.

Let $b(t):\R^+\rightarrow (0,+\infty)$ be an increasing function verifying
\begin{equation}\label{eq: scale}
\lim_{t\rightarrow \infty}b(t)=+\infty, \ \ \ \lim_{t\rightarrow \infty}\frac{b(t)}{\sqrt t}=0,
\end{equation}
define
\begin{equation}\label{eq: M measure}\mathfrak{M}_t:=\frac{1}{b(t)\sqrt t}\int_0^t(\delta_{X_s}-\pi)\dif s.\end{equation}
Then {\it moderate deviations} of $\mathfrak L_t$ from its {\it asymptotic limit} $\pi$ is to estimate
\begin{equation}\label{eq: moderate}
\PP_{\mu}\left(\mathfrak M_t \in A\right),
\end{equation}
where $A$ is some measurable set in  $(\mathcal M_b(H), \tau)$, a given domain of deviation. Here $\PP_{\mu}$ is the probability measure of the system $X$ with initial measure $\mu$.
When $b(t)=1$, this becomes an estimation of the central limit theorem; and when $b(t)=\sqrt t$, it is exactly the large deviations.
$b(t)$ satisfying \eqref{eq: scale} is between those two scalings, called {\it scaling of moderate deviations}, see \cite{DZ}.

Now we are at the position to state our MDP result:
\begin{thm}\label{thm main 2}
In the context of Theorem \ref{thm main 1}, for any initial measure $\mu$ verifying $\mu(\Psi)<+\infty$, the measure $\PP_{\mu}(\mathfrak M_t\in\cdot)$ satisfies the large deviation principle w.r.t. the $\tau$-topology with speed $b^2(t)$ and the rate function
 \begin{equation}
 I(\nu):=\sup\left\{\int f\dif \nu-\frac12\sigma^2(f);f\in b\mathcal B(H) \right\}, \ \  \ \forall \nu\in M_b(H)
 \end{equation}
where
\begin{equation}
\sigma^2(f)=\lim_{t\rightarrow \infty}\frac1t\E^{\pi}\left(\int_0^t (f(X_s)-\pi(f))\dif s \right)^2
\end{equation}
exists in $\RR$ for every $f\in B_{\Psi}\supset b\mathcal B(H)$. More precisely, the following three properties hold:
\begin{itemize}
  \item[(a1)] for any $a\ge0$, $\{\nu\in \mathcal M_b(H); I(\nu)\le a \}$ is compact in  $(\mathcal M_b(H),\tau)$;
  \item[(a2)]$($the upper bound$)$ for any  closed set $F$ in $(\mathcal M_b(H), \tau)$,
   $$
   \limsup_{T\rightarrow \infty}\frac1{b^2(T)}\log \mathbb P_{\mu}(\mathfrak M_T\in F)\le -\inf_F I;
   $$
    \item[(a3)] $($the lower bound$)$ for any open set $G$ in $(\mathcal M_b(H), \tau)$,
   $$
   \liminf_{T\rightarrow \infty}\frac1{b^2(T)}\log \mathbb P_{\mu}(\mathfrak M_T\in G)\ge -\inf_G I.
   $$
\end{itemize}

\end{thm}

\section{ Irreducibility in $H$  } \label{section irre}

In this section, we shall prove that $X=\{X_t^x\}_{t\ge0,x\in H}$ in the system \eqref{e:XEqn} is irreducible in $H$. Together with the strong Feller property established in \cite[Theorem 6.1]{Xu13}, this gives another proof to the existence  of at most one invariant measure by classical Doob's Theorem.

\vskip0.3cm



\subsection{Irreducibility of stochastic convolution}
Let us first consider the following  Ornstein-Uhlenbeck process:
\Be\label{e:OUAlp}
\dif Z_t+A Z_t \dif t= \dif L_t, \ \ \ Z_0=0,
\Ee
where $L_t=\sum_{k \in \Z_*} \beta_k l_k(t) e_k$ is an   $\alpha$-stable process on $H$. It is well known that
$$
Z_t=\int_0^t e^{-A(t-s)} \dif L_s=\sum_{k \in \Z_{*}} z_{k}(t) e_k,
$$
where $$z_{k}(t)=\int_0^t e^{-\gamma_k(t-s)}
\beta_k \dif l_k(s).$$
\vskip0.3cm
The following maximal inequality can be found  in \cite[Lemma 3.1]{Xu13}.
\begin{lem} \label{l:ZEst}  For  any $T>0, 0 \leq \theta<\beta-\frac 1{2 \alpha}$ and all $0<p<\alpha$, we have
$$
\E \sup_{0 \leq t \le T}\|A^\theta Z_t\|^p_{H} \le CT^{p/\alpha},
$$
where $C$ depends on $\alpha,\theta, \beta, p$.
\end{lem}

 The following lemma  is concerned with  the support of  the  distribution of $\big(\{Z_t\}_{0 \le t \le T},Z_T\big)$.
\begin{lem} \label{l:SupZt}
For  any $T>0, 0< p<\infty$, the random variable $\left(\{Z_t\}_{0 \le t \le T},Z_T\right)$ has a full support in $L^p([0,T];V)\times V$. More precisely, for any $\phi\in L^p([0,T];V), a\in V, \e>0$,
\Bes
\PP\left(\int_0^T \|Z_t-\phi_t\|_V^p \dif t+\|Z_T-a\|_V<\e\right)>0.
\Ees
\end{lem}

\begin{proof}
%
First,  by Lemma \ref{l:ZEst}, we have $Z \in L^\infty([0,T];V)$, a.s.
For  any $N \in \N$, let $H_N$ be the Hilbert space spanned by $\{e_k\}_{1 \le k \le N}$, and let
$\pi_N: H \rightarrow H_N$ be the orthogonal projection.  Notice that $\pi_N$ is also an orthogonal projection in $V$.
Define
$$\pi^N=I-\pi_N, \ \ \ \ H^N=\pi^N H.$$
By  the independence of $\pi_N Z$ and $\pi^N Z$, for any $\phi_t\in L^p([0,T];V), a\in V$, we have
\begin{align*}
&  \PP\left(\int_0^T \|Z_t-\phi_t\|_V^p \dif t+\|Z_T-a\|_V<\e\right) \\
\ge &\PP\bigg(\int_0^T \|\pi_N(Z_t-\phi_t)\|_V^p \dif t+\|\pi_N(Z_T-a)\|_V<\frac{\e}{2^{p+1}}, \\
& \ \ \ \ \int_0^T \|\pi^N(Z_t-\phi_t)\|_V^p \dif t+\|\pi^N(Z_T-a)\|_V<\frac{\e}{2^{p+1}}\bigg) \\
=&\PP\bigg(\int_0^T \|\pi_N(Z_t-\phi_t)\|_V^p \dif t+\|\pi_N(Z_T-a)\|_V<\frac{\e}{2^{p+1}}\bigg) \\
&   \times \PP\left(\int_0^T \|\pi^N(Z_t-\phi_t)\|_V^p \dif t+\|\pi^N(Z_T-a)\|_V<\frac{\e}{2^{p+1}}\right). \\
\end{align*}
By the same argument as in the Section 4.2 of \cite{PZ11}, we obtain
$$\PP\left(\int_0^T \|\pi_N(Z_t-\phi_t)\|_V^p \dif t+\|\pi_N(Z_T-a)\|_V<\frac{\e}{2^{p+1}}\right)>0.$$
To finish the proof, it suffices to show
$$
\PP\left(\int_0^T \|\pi^N(Z_t-\phi_t)\|_V^p \dif t+\|\pi^N(Z_T-a)\|_V<\frac{\e}{2^{p+1}}\right)>0.
$$
 For any $\theta \in (\frac 12, \beta-\frac{1}{2 \alpha})$,
by Lemma \ref{l:ZEst} (with $p=1$ therein), the spectral gap inequality and Chebyshev inequality, we have for any $\eta$
\begin{align*}
\PP\left(\sup_{0 \le t \le T} \|\pi^N Z_t\|_V \le {\eta}\right) &=1-\PP\left(\sup_{0 \le t \le T} \|\pi^N Z_t\|_V>{\eta}\right) \notag \\
& \ge 1-\PP\left(\sup_{0 \le t \le T} \|\pi^N A^{\theta} Z_t\|_H>{\eta} \gamma^{\theta-\frac 12}_N\right) \notag \\
& \ge 1-\PP\left(\sup_{0 \le t \le T} \|A^{\theta} Z_t\|_H>{\eta} \gamma^{\theta-\frac 12}_N\right)  \notag \\
& \ge 1-C_{\alpha, \beta, T}{\eta}^{-1} \gamma_N^{\frac12-\theta},\notag
\end{align*}
where $C_{\alpha, \beta, T}$ depends on ${\alpha, \beta, T}$. By the previous inequality, as long as $N$ (depending on $\e,p, \phi$) is sufficiently large,  we have
$$
\PP\left(\sup_{0 \le t \le T} \|\pi^N Z_t\|_V \le \frac{\e}{2^{2p+2}}\right)>0,
$$
and
$$
 \int_0^T\|\pi^N\phi_t\|_V^p \dif t+\|\pi^Na\|_V<\frac{\e}{2^{2p+2}}.
$$
Hence,
\Bes
\begin{split}
&\PP\left(\int_0^T \|\pi^N(Z_t-\phi_t)\|_V^p \dif t+\|\pi^N(Z_T-a)\|_V<\frac{\e}{2^{p+1}}\right)\\
\ge & \PP\left(\int_0^T \|\pi^NZ_t\|_V^p \dif t+\|\pi^NZ_T\|_V<\frac{\e}{2^{2p+2}}, \int_0^T\|\pi^N\phi_t\|_V^p \dif t+\|\pi^Na\|_V<\frac{\e}{2^{2p+2}}\right)\\
=& \PP\left(\int_0^T \|\pi^NZ_t\|_V^p \dif t+\|\pi^NZ_T\|_V<\frac{\e}{2^{2p+2}}\right) \\
>&0.
\end{split}
\Ees
The proof is complete.
\end{proof}
\vskip0.3cm

\subsection{A control problem for the deterministic system}
Consider the deterministic system in $H$,
\begin{equation}\label{e: deterministic}
\partial_t x(t)+Ax(t)=N(x(t))+u(t), \ \ \ x(0)=x_0,
\end{equation}
where $u\in L^2([0,T];V)$. By using the similar argument in the proof of Lemma 4.2 in \cite{Xu13},  for every $x(0)=x_0\in H, u\in L^2([0,T];V)$, the system \eqref{e: deterministic} admits a unique solution $x(\cdot)\in C([0, T];H)\cap C((0, T];V)$.  Moreover, $\{x(t)\}_{t\in[0,T]}$ has the following form:
\begin{equation}\label{e: solu deter}
x(t)=e^{-At} x_0+\int_0^t e^{-A(t-s)} N(x(s))\dif s+\int_0^t e^{-A(t-s)} u(s)\dif s, \ \ \ \forall \ t\in[0,T].
\end{equation}

 Next, we shall prove that the deterministic system is approximately controllable in time $T>0$.
\begin{lem} \label{l:AppCon}
For any $T>0, \e>0, a \in V$, there exists some $u \in L^\infty([0,T];V)$ such that the system \eqref{e: deterministic} satisfies that
\begin{equation*}
\|x(T)-a\|_V < \e.
\end{equation*}
\end{lem}

\begin{proof}
We shall prove the lemma by the following three steps.

\emph{Step 1. Regularization.} For any ${t_0}\in (0, T]$,  let $u(t)=0$ for all $t \in [0,{t_0}]$. Then the system \eqref{e: deterministic} admits a unique solution $x(\cdot)\in C([0, {t_0}];H)\cap C((0, {t_0}];V)$ with the following form:
\begin{equation*}
x(t)=e^{-At} x_0+\int_0^t e^{-A(t-s)} N(x(s))\dif s, \ \ \ \forall 0<t\le {t_0}.
\end{equation*}

\emph{Step 2.  Approximation at time $T$ and Linear interpolation.} For any $a\in V, \e>0$,  there exists a constant $\theta>0$ such that $$
\|e^{-\theta A}a-a\|_V\le \e.
$$Setting $x(t)=\frac{t-{t_0}}{T-{t_0}} e^{-\theta A}a+\frac{T-t}{T-{t_0}} x({t_0})$ for all $t \in [{t_0}, T]$. Then $x(\cdot)\in C((0, T];V)$.  By  \eqref{e: deterministic}, we have
$$u(t)=\frac{e^{-\theta A}a-x({t_0})}{T-{t_0}}-Ax(t)-N(x(t)),\ \ \ \  \forall t\in[{t_0}, T].$$

\emph{Step 3. } It remains to show that  $u\in L^{\infty}([0,T]; V)$. By \eqref{e:eAEst}, \eqref{e:NVEst} and
the constructions of $\{x(t)\}_{t\in[0,T]}$ and $\{u(t)\}_{t\in[0,T]}$ above, it is sufficient to show that $Ax({t_0})\in V$.
 For any $t\in [{t_0}/2,{t_0}]$,
\begin{align*}
x(t)=&e^{-(t-{t_0}/3)A}x({t_0}/3)+\int_{{t_0}/3}^t e^{-(t-s)A}N(x(s))\dif s \\
=&e^{-(t-{t_0}/2)A}x({t_0}/2)+\int_{{t_0}/2}^t e^{-(t-s)A}N(x(s))\dif s.
\end{align*}
 By \eqref{e:eAEst}, we have for $t>\frac{{t_0}}{2}$,
\begin{align} \label{e: Ax}
\|Ax(t)\|_H= & \|Ae^{-(t-{t_0}/3)A}x({t_0}/3)\|_H +\left\|A\int_{{t_0}/3}^t e^{-(t-s)A}N(x(s))\dif s\right\|_H\notag \\
\le & \|Ae^{-(t-{t_0}/3)A}x({t_0}/3)\|_H+\int_{{t_0}/3}^t \|A^{1/2}e^{-(t-s)A}\|\cdot\|A^{1/2}N(x(s))\|_H\dif s\notag \\
\le & C_1(t-{t_0}/3)^{-1}\|x({t_0}/3)\|_H+\int_{{t_0}/3}^t C_{1/2}{(t-s)^{-1/2}}\|N(x(s))\|_V\dif s.
\end{align}
Since $x(\cdot)\in C((0,{t_0}]; V)$, $\sup_{t\in[{t_0}/3,{t_0}]}\|x(s)\|_V<\infty$. Together with \eqref{e:NVEst} and \eqref{e: Ax}, we obtain that
$$
\sup_{t\in [{t_0}/2,{t_0}]}\|Ax(t)\|_H<\infty.
$$
By the previous inequality and  \eqref{e: NAH}, we have
\Bes
\begin{split}
\|A^{3/2}x({t_0})\|_H= & \left\|A^{3/2}e^{({t_0}/2)A}x({t_0}/2) +A^{3/2}\int_{{t_0}/2}^{{t_0}} e^{-(t-s)A}N(x(s))\dif s\right\|_H\\
\le & \|A^{3/2}e^{-({t_0}/2)A}x({t_0}/2)\|_H+\int_{{t_0}/2}^{{t_0}} \|A^{1/2}e^{-({t_0}-s)A}\|\cdot\|AN(x(s))\|_H\dif s\\
\le & C_{3/2}({t_0}/2)^{-3/2}\|x({t_0}/2)\|_H+\int_{{t_0}/2}^{{t_0}} C_{1/2}{({t_0}-s)^{-1/2}}\|AN(x(s))\|_H\dif s\\
\le & C_{3/2}({t_0}/2)^{-3/2}\|x({t_0}/2)\|_H+C\sup_{s\in[{t_0}/2,{t_0}]}(1+\|x(s)\|_V^2)\cdot(1+ \|Ax(s)\|_H^2)\\
<&\infty,
\end{split}
\Ees
which means that $Ax({t_0})\in V$. The proof is complete.
\end{proof}

\vskip0.3cm
\subsection{Irreducibility   in H}

 Now we  prove Theorem \ref{thm irred}  by following the idea  in \cite[Theorem 5.4]{PZ11}.

\begin{proof}[Proof of Theorem \ref{thm irred}] For any $x_0\in H, t>0$,  we have $X_t^{x_0}\in V$ a.s. by Theorem \ref{Thm Xu 13}.  Since $X$ is Markov in $H$,  for any $a\in H, T>0,\e>0$,
\begin{align*}
\PP \left(\|X_T^{x_0}-a\|_H<\e\right)=&\int_V \PP \left(\|X_T^{x_0}-a\|_H<\e| X_t^{x_0}=v\right)\PP(X_t^{x_0}\in \dif v)\\
=&\int_V \PP \left(\|X_{T-t}^{v}-a\|_H<\e\right)\PP(X_t^{x_0}\in \dif v).
\end{align*}
 To prove that
\begin{equation*}\label{e: ir}
\PP \left(\|X_T^{x_0}-a\|_H<\e\right)>0,
\end{equation*}
 it is sufficient to prove that for any $T>0$,
$$
\PP \left(\|X_T^{x_0}-a\|_H<\e\right)>0 \ \ \ \ \text{for all } x_0\in V.
$$
   Next, we prove the theorem under the assumption of the initial value $x_0\in V$ in  the following two steps.
\vskip0.3cm
\emph{Step 1}.  For any $a\in H,\e>0$, there exists some $\theta>0$ such that $e^{-\theta A}a\in V$ and
\begin{equation}\label{e: irre 1}
\|a-e^{-\theta A}a\|_H\le \frac \e 4.
\end{equation}
For any $T>0$, by Lemma \ref{l:AppCon} and the spectral gap inequality, there exists some
$u \in L^\infty([0,T];V)$ such that the system
\begin{equation*}
\dot x+Ax=N(x)+u, \ \ \ x(0)=x_0,
\end{equation*}
satisfies that
\begin{equation}\label{e: irre 2}
\|x(T)-e^{-\theta A}a\|_H \le\|x(T)-e^{-\theta A}a\|_V< \frac\e 4.
\end{equation}
Putting  \eqref{e: irre 1} and \eqref{e: irre 2} together, we have
\begin{equation}\label{e: irre 3}
\|x(T)-a\|_H< \frac\e 2.
\end{equation}

\vskip0.3cm

\emph{Step 2:}
We shall consider the systems \eqref{e:Lin} and \eqref{e:LinS}
as follows:
 \begin{equation}\label{e:Lin}
    \begin{cases}
 \dot z+A z=u, \ \ \ \ z(0)=0, \\
\dot y+A y=N(y+z), \ \ \ \ y(0)=x_0\in V,
\end{cases} \end{equation}
and
 \begin{equation}\label{e:LinS}
    \begin{cases}
\dif Z_t+A Z_t \dif t= \dif L_t, \ \ \ \ Z_0=0;\\
\dif Y_t+A Y_t\dif t=N(Y_t+Z_t)\dif t, \ \ \ \ Y_0=x_0\in V.
 \end{cases} \end{equation}
By the arguments in the proof of Lemma 4.2 in \cite{Xu13}, for any $x_0\in V, u\in L^2([0,T];V)$, the systems \eqref{e:Lin} and \eqref{e:LinS} admit the unique solutions $(y(\cdot), z(\cdot)) \in  C([0,T];V)^2$ and $(Y_{\cdot}, Z_{\cdot}) \in   C([0,T];V)^2$, a.s.
 Furthermore, denote
$$
x(t)=y(t)+z(t),\ \ \ X_t=Y_t+Z_t,\ \ \ \ \forall t\ge0.
$$
For any $0 \le t\le T$,
\begin{align*}
 &  \|Y_t-y(t)\|_H^2 +2\int_{0}^t \|Y_s-y(s)\|^2_V \dif s \\
=&2 \int_{0}^t \Ll Y_s-y(s), N(X_s)-N(x(s))\Rr_H \dif s \\
=& 2\int_{0}^t \|Y_s-y(s)\|^2_H \dif s+2 \int_0^t \Ll Y_s-y(s),Z_s-z(s)\Rr_H \dif s \\
& -2 \int_{0}^t \left\Ll Y_s-y(s), X^3_s-x^3(s)\right\Rr_H \dif s.
 \end{align*}
Let us estimate the third term of the right hand side. Denoting $\Delta Y_s=Y_s-y(s)$ and
$\Delta Z_s=Z_s-z(s)$, we have
\begin{equation*}
\begin{split}
 &\int_{0}^t \left\Ll Y_s-y(s), X^3_s-x^3(s)\right\Rr_H \dif s\\
=&  \int_{0}^t \left\Ll \Delta Y_s, [\Delta Y_s+\Delta Z_s+x(s)]^3-x^3(s)\right\Rr_H \dif s\\
=&\int_{0}^t \langle \Delta Y_s,[\Delta Y_s+\Delta Z_s]^3+3[\Delta Y_s+\Delta Z_s]^2 x(s)+3[\Delta Y_s+\Delta Z_s] x^2(s) \rangle_H \dif s\\
=& \int_{0}^t \langle \Delta Y_s, (\Delta Y_s)^3  +3 (\Delta Y_s)^2 \Delta Z_s +3 \Delta Y_s (\Delta Z_s)^2 +(\Delta Z_s)^3\rangle_H \dif s \\
&+3\int_{0}^t \left\langle \Delta Y_s,[(\Delta Y_s)^2+2 \Delta Y_s \Delta Z_s+ (\Delta Z_s)^2] x(s)\right\rangle_H\dif s +3\int_{0}^t \langle \Delta Y_s,[\Delta Y_s+\Delta Z_s] x^2(s)\rangle _H\dif s.
\end{split}
\end{equation*}
Since $\frac 34  (\Delta Y_s)^4+3 (\Delta Y_s)^3 x(s)+3(\Delta Y_s)^2x^2(s) \ge 0$, from the above relation we have
\begin{align*}
 &\int_{0}^t \left\Ll Y_s-y(s), X^3_s-x^3(s)\right\Rr_H \dif s \\
\ge& \int_{0}^t \langle \Delta Y_s,3 (\Delta Y_s)^2 \Delta Z_s+3 \Delta Y_s (\Delta Z_s)^2 +(\Delta Z_s)^3\rangle_H \dif s\\
&+3\int_{0}^t \langle \Delta Y_s,[2 \Delta Y_s \Delta Z_s+(\Delta Z_s)^2] x(s) \Rr_H \dif s\\
&+3\int_{0}^t \langle \Delta Y_s, \Delta Z_s x^2(s)\rangle_H \dif s+\frac14 \int_{0}^t \|\Delta Y_s\|_{L^4}^4\dif s.
\end{align*}
Using the following Young inequalities: for all $y,z\in L^4(\T; \R)$,
\begin{equation*}
\begin{split}
& |\Ll  y, z \Rr_H|=\left|\int_{\mathbb T} y(\xi) z (\xi) \dif \xi\right| \le \frac{\int_{\mathbb T} y^4(\xi) \dif \xi}{80}+ C\int_{\mathbb T} z^{\frac 43}(\xi) \dif \xi, \\
& |\Ll  y^2, z \Rr_H|=\left|\int_{\mathbb T} y^2(\xi) z (\xi) \dif \xi\right| \le \frac{\int_{\mathbb T} y^4(\xi) \dif \xi}{80}+C\int_{\mathbb T} z^2(\xi) \dif \xi. \\
& |\Ll  y^3, z  \Rr_H|=\left|\int_{\mathbb T} y^3(\xi) z(\xi) \dif \xi\right| \le \frac{ \int_{\mathbb T} y^4(\xi) \dif \xi}{80}+ C\int_{\mathbb T} z^4(\xi) \dif \xi,
\end{split}
\end{equation*}
and the H\"older inequality, we further get
\begin{align*}
 &\int_{0}^t \left\Ll Y_s-y(s), X^3_s-x^3(s)\right\Rr_H \dif s \\
\ge&  \frac{1}{80}\int_{0}^t \|\Delta Y_s\|_{L^4}^4\dif s-7C\int_{0}^t \|\Delta Z_s\|_{L^4}^4 \dif s\\
&-6C\int_{0}^t\|\Delta Z_s x(s)\|_{L^2}^2\dif s-3C\int_{0}^t\|(\Delta Z_s)^2 x(s)\|_{L^{\frac43}}^{\frac43}\dif s\\
&-3C\int_{0}^t\|\Delta Z_s x^2(s)\|_{L^{\frac43}}^{\frac43}\dif s\\
\ge&  \frac{1}{80}\int_{0}^t \|\Delta Y_s\|_{L^4}^4\dif s-7C\int_{0}^t \|\Delta Z_s\|_{L^4}^4 \dif s\\
&-6C\int_{0}^t\|\Delta Z_s\|_{L^4}^{2}\|x(s)\|_{L^4}^{2}\dif s-3C\int_{0}^t\|\Delta Z_s\|_{L^4}^{\frac83}\|x(s)\|_{L^4}^{\frac43}\dif s\\
&-3C\int_{0}^t\|\Delta Z_s\|_{L^4}^{\frac43}\|x(s)\|_{L^4}^{\frac83}\dif s.
\end{align*}
 Since $x(t)=y(t)+z(t)\in   C([0, T]; V)$, by \eqref{e:L4}, there exists a constant $C_T$ such that
 $$
\sup_{s\in [0, T]} \|y(s)+z(s)\|_{L^4}\le  \sup_{s\in [0, T]}\|y(s)+z(s)\|_{H}^{\frac12}\cdot\|y(s)+z(s)\|_{V}^{\frac12}\le C_T.
 $$
Consequently,  there is some constant $C_T>0$ satisfying that
\begin{equation*}
\begin{split}
& \|Y_t-y(t)\|_H^2+2\int_{0}^t \|Y_s-y(s)\|^2_V \dif s \\
\le& 3\int_{0}^t \|Y_s-y(s)\|^2_H \dif s+\int_{0}^t \|Z_s-z(s)\|^2_H \dif s\\
& +C_T\int_{0}^t\Big( \|Z_s-z(s)\|_{L^4}^4+\|Z_s-z(s)\|_{L^4}^{2} +\|Z_s-z(s)\|_{L^4}^{\frac83} +\|Z_s-z(s)\|_{L^4}^{\frac43} ds\Big) \dif s.
 \end{split}
\end{equation*}
Therefore, by the spectral gap inequality and Gronwall's inequality, we have
\begin{equation}\label{e: Gron}
\begin{split}
\|Y_T-y(T)\|_H^2 \le  C_T \sum_{i\in \Lambda} \int_{0}^T \|Z_s-z(s)\|_{V}^{i }  \dif s,
\end{split}
\end{equation}
where $\Lambda:=\{  4/3, 2, 8/3,4\}$. This inequality, together with Lemma \ref{l:SupZt}, \eqref{e: irre 3}, implies
\begin{equation*} 
\begin{split}
&\PP \left(\|X_T-a\|_H<\e\right)\\
=&\PP \left(\|Y_T-y(T)+Z_T-z(T)+x(T)-a\|_H<\e\right) \\
\ge&\PP \left(\|Y_T-y(T)\|_H \le \e/4, \|Z_T-z(T)\|_H \le  \e/4, \|x(T)-a\|_H<\e/2\right) \\
=&\PP \left(\|Y_T-y(T)\|_H \le \e/4, \|Z_T-z(T)\|_H \le  \e/4\right)\\
\ge& \PP \left(   \sum_{i\in \Lambda} \int_{0}^T \|Z_s-z(s)\|_{V}^{i }  \dif s+\|Z_T-z(T)\|_V \le C_{T, \e}\right)\\
>&0.
\end{split}
\end{equation*}
The proof is complete.
\end{proof}

\section{The proofs of Theorems \ref{thm main 1} and   \ref{thm main 2}}
\subsection{Several general results for strong Feller Markov processes}

In this subsection, we recall some general results about moderate deviations and exponential convergence for general strong Feller Markov processes, borrowed from \cite{Wu01}.
 \vskip0.3cm
We say that a measurable function $f:H\rightarrow \R$ belongs to the extended domain $\mathbf{D}_e(\LLL)$ of the generator $\LLL$ of $(P_t)$,
if there is a measurable function $g:H\rightarrow\R$ so that $\int_0^t|g|(X_s)ds<+\infty, \forall t>0, \PP_x-a.s.$ and
$$
f(X_t)-f(X_0)-\int_0^t g(X_s)ds, \ \ \ t\ge0
$$
is a c\`{a}dl\`{a}g $\PP_x$-local martingale for all $x\in H$. In that case, $g:=\LLL f$.

\begin{thm}\label{thm DWT}$($\cite[Theorem 5.2c]{DWT} or \cite[Theorem 2.4]{Wu01}$)$ Assume that the process $(X_t)$ is strong Feller, irreducible and aperiodic $($see \cite{DWT} for definition, that is the case if $P_T(\cdot, K)>0$ over $H$ for some compact $K$ verifying $\pi(K)>0$$)$.  If there are some continuous function $1\le \Psi\in \mathbf{D}_e(\mcl L)$,   compact subset $K\subset H$ and constants $\varepsilon, C>0$ such that
\begin{align}\label{Lyapunov}
 -\frac{\LLL \Psi}{\Psi}\ge\e \1_{K^c}-C\1_{K},
\end{align}
then there is a unique invariant probability measure $\pi$ satisfying
 \begin{align*}
\int \Psi \dif \pi<+\infty,
 \end{align*}
and there are some constants $\theta>0$ and $0<\rho<1$ such that for all $t\ge0$,
\begin{equation}\label{eq: exponential}
\sup_{|f|\le \Psi}\left|P_tf(z)-\int f\dif \pi \right|\le \theta \Psi(z)\cdot \rho^t\ \ \  z\in H.
\end{equation}
\end{thm}

\vskip0.3cm

For the measure-valued process $\mathfrak M_t$ defined in \eqref{eq: M measure}, we have the following large deviations result.
\begin{thm}\label{thm Wu}$($\cite[Theorem 2.6]{Wu01}$)$ Assume that the process $(X_t)$ is strong Feller, irreducible, aperiodic and satisfies \eqref{Lyapunov}. For any initial measure $\mu$ verifying $\mu(\Psi)<+\infty$, the measure  $\PP_{\mu}(\mathfrak M_t\in\cdot)$ satisfies the large deviation principle w.r.t. the $\tau$-topology with  the speed $b^2(t)$ and  the rate function
 \begin{equation}
 I(\nu):=\sup\left\{\int f\dif \nu-\frac12\sigma^2(f);f\in b\mathcal B (H) \right\}, \ \  \ \forall \nu\in \mathcal M_b(H)
 \end{equation}
where
\begin{equation}
\sigma^2(f)=\lim_{t\rightarrow \infty}\frac1t\E_{\pi}\left(\int_0^t (f(X_s)-\pi(f))\dif s \right)^2
\end{equation}
exists in $\RR$ for every $f\in B_{\Psi}\supset b\mathcal B(H)$.

\end{thm}

\subsection{The Proofs of main results}
In this subsection, we shall prove Theorems \ref{thm main 1} and  \ref{thm main 2} based on the above theorems. The main technique is to construct a suitable Lyapunov test function.

\begin{proof}[Proofs of  Theorems \ref{thm main 1} and   \ref{thm main 2}]  By Theorems \ref{Thm Xu 13} and \ref{thm irred}, the system \eqref{e:XEqn} is  strong Feller, irreducible and aperiodic in $H$. Indeed, as the invariant  measure $\pi$  is supported on $V$, there exists a bounded  closed ball $F\subset V$ satisfying $\pi(F)>0$. Notice that $F$ is compact in $H$.
Since the system $X$ is strong Feller and  irreducible in $H$,  by \cite[Theorem 4.2.1]{DPZ96}, the invariant measure $\pi$ is equivalent to all measures $P_t(x,\cdot)$, for all $x\in H, t>0$.  Consequently, $P_t(x, F)>0$ for all $x\in H, t>0$, which implies that the system is aperiodic.


By Theorems \ref{thm DWT} and   \ref{thm Wu}, we now construct a suitable Lyapunov function $\Psi$ satisfying \eqref{Lyapunov}.
Take
\begin{align}\label{e Lyap}
\Psi(x):=(M+\|x\|_H^2)^{\frac12},
\end{align}
where $M$ is a large constant to be determined later. By Lemma \ref{l:DeL} below,  we have $\Psi \in \mathbf{D}_e(\mcl L)$.

Recall $x^m=\pi_m x$, we have
\begin{align}\label{eq: I}
\LLL\Psi(x^m)=&\langle -Ax^m, D\Psi(x^m)\rangle+\langle N(x^m),D\Psi(x^m)\rangle\notag \\
 &+\sum_{|i| \le m}\int_{|y_i|\le 1}\left[\Psi(x^m+\beta_iy_i e_i)-\Psi(x^m)-\beta_iy_i D_{e_i}\Psi(x^m)\right]\nu(\dif y_i)\notag \\
 &+\sum_{|i| \le m}\int_{|y_i|> 1}\left[\Psi(x^m+\beta_iy_i e_i)-\Psi(x^m) \right]\nu(\dif y_i)\notag \\
 =:&J^m_1+J^m_2+J^m_3+J^m_4.
\end{align}
Here
$$
  \ D_{e_i}\Psi(x^m):=\frac{x_i}{\Psi(x^m)},\ \ \ \ D\Psi(x^m):=\sum_{|i| \le m} (D_{e_i}\Psi(x^m))e_i=\frac{x^m}{\Psi(x^m)}.
$$

For the first term,  using the  integration  by parts formula, we have
\begin{align}\label{eq: I1}
\langle -Ax^m, D\Psi(x^m)\rangle=\langle -Ax^m, \frac{x^m}{\Psi(x^m)} \rangle=-\frac{\|x^m\|_V^2}{\Psi(x^m)}.
\end{align}

For the second   term, by  (2.5) in \cite{Xu13} (note that $N(x)$ here equals $-N(x)$ in \cite{Xu13}), we have
\begin{align}\label{eq: I2}
\langle N(x^m), D\Psi(x^m) \rangle=\langle N(x^m), \frac{x^m}{\Psi(x^m)} \rangle\le \frac1{4\Psi(x^m)}.
\end{align}

\noindent For any $h\in H$,
\begin{align*}
|\langle h, D^2\Psi(x^m) h\rangle|=\frac{\|h^m\|_H^2}{\sqrt{M+\|x^m\|^2}}-\frac{|\langle h^m, x^m\rangle |^2}{(M+\|x^m\|_H^2)^{\frac32}}\le\frac{\|h\|_H^2}{\sqrt{M+\|x^m\|_H^2}}.
\end{align*}
This inequality, together with Taylor's formula, implies that
\begin{align*}
|\Psi(x^m+\beta_iy_i e_i)-\Psi(x^m)-\beta_iy_i D_{e_i}\Psi(x^m)|\le \frac{\beta_i^2 y_i^2}{\sqrt{M+\|x^m\|_H^2}}.
\end{align*}
Thus, for the third term, we have
\begin{align}\label{eq: I3}
 &\sum_{|i| \le m}\int_{|y_i|\le 1}\left|\Psi(x^m+\beta_iy_i e_i)-\Psi(x^m)-\beta_iy_i D_{e_i}\Psi(x^m)\right|\nu(\dif y_i)\notag \\
 \le & \sum_{|i| \le m}\int_{|y_i|\le 1}\frac{\beta_i^2 y_i^2}{\sqrt{M+\|x^m\|_H^2}}  \nu(\dif y_i)\notag \\
 =& \frac{\sum_{|i| \le m}\beta_i^2}{\sqrt{M+\|x^m\|_H^2}} \int_{|y|\le 1}\frac{ |y|^{1-\alpha}}{C_\alpha} \dif y\notag \\
=& \frac{2\sum_{|i| \le m}\beta_i^2}{C_\alpha(2-\alpha)\sqrt{M+\|x^m\|_H^2}}<\frac{2\sum_{i \in {\Z}_*}\beta_i^2}{C_\alpha(2-\alpha)\sqrt{M+\|x^m\|^2_H}}<+\infty,
\end{align}
where $\nu$ is the L\'evy measure of 1-dimensional $\alpha$-stable process and we have used the assumption of $\beta_i$ in (ii) in the last inequality.

By Taylor's formula again, there exists $\tl x^m\in H$ satisfying that
\begin{align*}
|\Psi(x^m+\beta_i y_i e_i)-\Psi(x^m)|=|\langle  D_{e_i}\Psi(\tl x^m), \beta_iy_ie_i\rangle|=\frac{\left|\tl x_i\beta_i y_i\right|}{\sqrt{M+\|\tl x^m\|_H^2}}\le |\beta_i y_i|.
\end{align*}
For the fourth term, we have
\begin{align}\label{eq: I4}
&\left|\sum_{|i| \le m}\int_{|y_i|> 1}\left[\Psi(x^m+\beta_iy_i e_i)-\Psi(x^m) \right]\nu(\dif y_i)\right|
\le\sum_{|i| \le m}\int_{|y_i|> 1}|\beta_i y_i| \nu(\dif y_i)\notag \\
=&\sum_{|i| \le m}\int_{|y_i|> 1} \frac{|\beta_i y_i|}{C_{\alpha}|y_i|^{1+\alpha}} \dif y_i
 =\frac{2\sum_{|i| \le m} |\beta_i|}{C_{\alpha}(\alpha-1)}<\frac{2\sum_{i \in {\Z}_*} |\beta_i|}{C_{\alpha}(\alpha-1)}<+\infty,
 \end{align}
where we have used the assumption of $\beta_i$ in (ii) again.

Putting \eqref{eq: I}-\eqref{eq: I4} together,  we obtain that for any $x\in H$,
\begin{align} \label{eq: II}
 -\frac{\LLL\Psi(x^m)}{\Psi(x^m)}
=&-\frac{J^m_1+J^m_2+J^m_3+J^m_4}{\Psi(x^m)} \notag \\
\ge& \frac{\|x^m\|_V^2}{M+\|x^m\|_H^2}-\frac{1}{4(M+\|x^m\|_H^2)}-\frac{2 \sum_{i \in {\Z}_*}\beta_i^2}{C_\alpha(2-\alpha)(M+\|x^m\|_H^2)}\notag\\
&-\frac{2 \sum_{i \in {\Z}_*} |\beta_i|}{C_{\alpha}(\alpha-1)\sqrt{M+\|x^m\|_H^2}}.
\end{align}
Let
\begin{align*}
K:=\{x\in V; \|x\|_V^2\le M\}.
\end{align*}
Then $K$ is compact in $H$.  By (3) in Theorem \ref{Thm Xu 13},  choose $M$ large enough such that   $\pi(K)>0$ and \begin{align}\label{eq: I5}
\frac{1}{4M}+\frac{2 \sum_{i \in {\Z}_*}\beta_i^2 }{C_\alpha(2-\alpha)M}+\frac{2\sum_{i \in {\Z}_*}|\beta_i|}{C_{\alpha}(\alpha-1)\sqrt M} \le \frac14, \ \ \ \forall x\in H.
\end{align}
Since $\lim_{m \rightarrow \infty} \Psi(x^m)=\Psi(x)$ and $\lim_{m \rightarrow \infty} \mcl L \Psi(x^m)$ has limit for $x \in V$, by the closable property of $\mcl L$, we immediately get $\mcl L \Psi(x)=\lim_{m \rightarrow \infty} \mcl L \Psi(x^m)$ and thus
\
\begin{equation}  \label{e:Ver1}
\begin{split}
&-\frac{\LLL\Psi(x)}{\Psi(x)}  \ge -\frac 14, \ \ \ \ \ \ x \in K.
\end{split}
\end{equation}

For any $x\in K^c$, by \eqref{eq: II} and \eqref{eq: I5},  we have
\begin{align*}
-\frac{\LLL\Psi(x^m)}{\Psi(x^m)}\ge  &   \frac{\|x^m\|_V^2}{M+\|x^m\|_H^2}-\frac{1}{4}\ge \frac14.
\end{align*}
This implies
\begin{align} \label{e:Ver2}
-\frac{\LLL\Psi(x)}{\Psi(x)}\ge  &   \frac{\|x^m\|_V^2}{M+\|x^m\|_H^2}-\frac{1}{4}\ge \frac14 \ \ \ \ \ \ \ \forall \ x \in V \cap K^c,
\end{align}
and
\begin{align}   \label{e:Ver3}
-\frac{\LLL\Psi(x)}{\Psi(x)}=\infty \ \ \ \ \ \ \ \forall \ x \in H \setminus (V \cap K^c).
\end{align}

Putting \eqref{e:Ver1}-\eqref{e:Ver3} together, we immediately obtain
\begin{align*}
-\frac{\LLL\Psi(x)}{\Psi(x)}\ge  &  \frac14 \1_{K^c }-\frac14 \1_{K}.
\end{align*}
The proof is complete.
\end{proof}
\vskip0.3cm

\begin{lem}\label{l:DeL} For $\Psi$ defined in \eqref{e Lyap}, we have $\Psi \in  \mathbf{D}_e(\mcl L)$.
\end{lem}

Before proving $\Psi \in  \mathbf{D}_e(\mcl L)$,
let us first briefly review some well known facts about $\alpha$-stable process for using It\^{o} formula. Let  $\{l_j(t)\}_{j \ge 1}$ be a sequence of
i.i.d. 1-dimensional $\alpha$-stable processes. The Poisson random measure associated with $l_j(t)$ is defined by
$$
N^{(j)}(t,\Gamma):=\sum_{s\in(0,t]}1_{\Gamma}(l_j(s)-l_j(s-)), \ \ \ \  \forall \ t>0 \ \ \forall \ \Gamma\in \mcl B(\R\setminus\{0\}).
$$
By L\'evy-It\^o's decomposition (cf. \cite[p.126, Theorem 2.4.16]{Ap09}), one has
$$
l_j(t)=\int_{|x|\leq 1}x\tilde N^{(j)}(t,\dif x)+\int_{|x|>1}x N^{(j)}(t,\dif x),
$$
where $\tl N^{(j)}$ is the compensated Poisson random measure defined by
$$
\tilde N^{(j)}(t,\Gamma)=N^{(j)}(t,\Gamma)-t\nu(\Gamma).
$$

\begin{proof}[Proof of Lemma \ref{l:DeL}]
The proof follows from \cite[Sect. 3]{DXZ11} in spirit.
Let $T>0$ be an arbitrary but finite number, we shall consider the stochastic system in $[0,T]$. Consider the Galerkin approximation of \eqref{e:XEqn}:
\Be \label{e:GalEqn}
\dif X^m_t+A X^m_t \dif t=N^m(X^m_t)dt+dL^m_t,
\ \ \ \ X^m_0=x^m,
\Ee
where $X^m_t=\pi_m X_t$, $N^m(X^m_t)=\pi_m[N(X^m_t)]$,
$L^m_t=\sum_{|k| \le m} \beta_k l_k(t) e_k$, $\pi_m$ is the orthogonal projection  defined in the proof of Lemma \ref{l:SupZt}.
By a standard argument (\cite{BLZ14,PrRo07}), for all $x \in W$ with $W=H$ or $W=V$ we have
\
\Be  \label{e:UniXm}
\E \left[\sup_{0 \le t \le T}\|X^{m}_t(x^m)\|_W\right] \le C_W(x,T),
\Ee
\Be  \label{e:UniXmX}
\lim_{m \rightarrow \infty} \E \left[\sup_{0 \le t \le T}\|X^{m}_t(x^m)-X_t(x)\|_W\right]=0,
\Ee
where $C_W(x,T)>0$ is finite.

Write
$\Psi(u):=(M+\|u\|_H^2)^{1/2}$ for all $u \in H$,
it follows from It\^{o} formula (\cite{Ap09}) that
\
\Be  \label{e:ItoM}
\Psi(X^m_t)-\Psi(x^m)+I^m_1(t)-I^m_2(t)=I^m_3(t)+I^m_4(t),
\Ee
where
\begin{align*}
& I^m_1(t):=\int^t_0\frac{\|X^m_s\|^2_V}{(M+\|X^m_s\|^2_H )^{\frac 12}} \dif s-\int^t_0\frac{\langle X^m_s, N(X^m_s)\rangle_H}{(M+\|X^m_s\|^2_H)^{\frac 12}} \dif s, \\
& I^m_2(t):=\sum_{|j| \le m}\int^t_0\!\!\!\int_{\R}\left[\Psi(X^m_s+y\beta_j e_j)-\Psi(X^m_s)
-\frac{\langle X^m_s,y\beta_je_j \rangle}{(M+\|X^m_s\|_H^2)^{\frac 12}} 1_{\{|y| \le 1\}}\right]\nu(\dif y) \dif s,\\
& I^m_3(t):=\sum_{|j| \le m}\int^t_0\!\!\!\int_{|y|\leq 1}[\Psi(X^m_s+y\beta_j e_j)-\Psi(X^m_s)]\tilde N^{(j)}(\dif s,\dif y),\\
& I^m_4(t):=\sum_{|j| \le m}\int^t_0\!\!\!\int_{|y|> 1}\left[\Psi(X^m_s+y\beta_j e_j)-\Psi(X^m_s)\right]\tl N^{(j)}(\dif s,\dif y).
\end{align*}
By a Taylor expansion argument similar to \eqref{eq: I3} and \eqref{eq: I4} below, for all $T>0$ we have
\Bes
\begin{split}
& \E\left[\sup_{0 \le t \le T} |I^m_2(t)|\right] \le CT, \\
& \E\left[\sup_{0 \le t \le T} |I^m_3(t)|^2\right] \le CT, \\
& \E\left[\sup_{0 \le t \le T} |I^m_4(t)|\right] \le CT.
\end{split}
\Ees
Moreover, by a Taylor expansion argument similar to \eqref{eq: I3} and \eqref{eq: I4} below again, we get that as $m_1 \rightarrow \infty, m_2 \rightarrow \infty$,
\Bes
\E \left[\sup_{0 \le t \le T} \left|I^{m_1}_2(t)-I^{m_2}_2(t)\right|\right] \rightarrow 0,
\Ees
\Bes
\E \left[\sup_{0 \le t \le T}\left|I^{m_1}_3(t)-I^{m_2}_3(t)\right|^2\right] \rightarrow 0,
\Ees
\Bes
\E\left[ \sup_{0 \le t \le T} \left|I^{m_1}_4(t)-I^{m_2}_4(t)\right|^p \right]\rightarrow 0,
\Ees
for all $1 \le p<\alpha$. Hence, there exist $I_2$, $I_3$ and
$I_4$ such that
\Be \label{e:I2mCon}
\lim_{m \rightarrow \infty} \E\left[\sup_{0 \le t \le T} \left|I^{m}_2(t)-I_2(t)\right|\right]=0,
\Ee
\Be  \label{e:I3mCon}
\lim_{m \rightarrow \infty} \E\left[ \sup_{0 \le t \le T} |I^{m}_3(t)-I_3(t)|^2\right]=0,
\Ee
\Be  \label{e:I4mCon}
\lim_{m \rightarrow \infty} \E \left[\sup_{0 \le t \le T} |I^{m}_4(t)-I_4(t)|^p\right]=0,
\Ee
where $I_2$, $I_3$ and $I_4$ have the same forms as
$I^m_2$, $I^m_3$ and $I^m_4$ but with $\sum_{|i| \le m}$ replaced by  $\sum_{i \in \mathbb{Z}_*}$ and $X^m$ replaced by $X$.
It is also easy to verify that $I_3$ is an $L^2$ martingale and that $I_4$ is an $L^1$ martingale.

Next we shall show below,  taking a subsequence if necessary, that
\Be  \label{e:I1mCon}
\lim_{m \rightarrow \infty} I^m_1(t)=I_1(t) \ \ \ \forall \ \   0 \le t \le T, \ \ \ \ a.s.,
\Ee
where $I_1(t)$ has the same form as $I^m_1(t)$ but with
$X^m$ replaced by $X$.
Collecting \eqref{e:I2mCon}-\eqref{e:I1mCon}, taking a subsequence if necessary and letting $m \rightarrow \infty$ in \eqref{e:ItoM}, we obtain
\Be
\Psi(X_t)-\Psi(x)+I_1(t)-I_2(t)=I_3(t)+I_4(t).
\Ee
Since $I_3$ and $I_4$ are $L_2$ and $L_1$ martingales respectively, taking
\Bes
\begin{split}
g(t)=&-\int_0^t\frac{\|X_s\|^2_V}{(M+\|X_s\|_H^2)^{\frac 12}} \dif s+\int^t_0\frac{\langle X_s, N(X_s)\rangle_H}{(M+\|X_s\|_H^2 )^{\frac 12}} \dif s \\
&+\sum_{j \in {\Z}_*}\int_0^t\int_{\R}\left[\Psi(X_s+y\beta_j e_j)-\Psi(X_s)
-\frac{\langle X_s,y\beta_je_j \rangle_H}{(M+\|X_s\|_H^2)^{\frac 12}} 1_{\{|y| \le 1\}}\right]\nu(\dif y) \dif s,
\end{split}
\Ees
we immediately verify that $\Psi \in \mathbf{D}_e(\mcl L)$
for $t \in [0,T]$. Since $T>0$ is arbitrary, $\Psi \in \mathbf{D}_e(\mcl L)$
for $t \in [0,\infty)$.

It remains to prove \eqref{e:I1mCon}. Taking a subsequence if necessary and letting $m \rightarrow \infty$ in \eqref{e:ItoM}, by Fatou lemma and the fact $\Ll x,N(x)\Rr \le \frac 14$ from \cite{Xu13} we have
\
\Bes
\begin{split}
& \ \ \ \E\big[\sup_{t\in[0,T]}(M+\|X_t\|_H^2)^{\frac 12}\big]
+\E\left[\int^T_0\frac{\|X_s\|^2_V}{(M+\|X_s\|_H^2)^{\frac 12}}\dif s\right] \leq (M+\|x\|_H^2)^{\frac 12}+CT+CT^{\frac 12}.
\end{split}
\Ees
This implies
\Be  \label{e:IntXsV2}
\int^t_0\frac{\|X_s\|^2_V}{(M+\|X_s\|_H^2)^{\frac 12}}\dif s<\infty \ \ \ \ \ \ \ \ \ \forall \ \ t \in [0,T],  \ a.s.
\Ee
It is easy to check
$\frac{\|X^m_s\|^2_V}{(M+\|X^m_s\|_H^2)^{\frac 12}}$ is increasing in $m$ for every $s>0$ and
$$\frac{\|X^m_s\|^2_V}{(M+\|X^m_s\|_H^2)^{\frac 12}} \le \frac{\|X_s\|^2_V}{(M+\|X_s\|_H^2 )^{\frac 12}} , \ \ \ \ s>0.$$
Hence, by \eqref{e:UniXmX} and the Lesbegue dominated convergence theorem,  taking a subsequence if necessary,  we get
\Be \label{e:Im11Con}
\lim_{m \rightarrow \infty} \int_0^t \frac{\|X^m_s\|^2_V}{(M+\|X^m_s\|_H^2)^{\frac 12}}\dif s=\int_0^t \frac{\|X_s\|^2_V}{(M+\|X_s\|_H^2 )^{\frac 12}} \dif s, \  \ \ \ \ a.s.
\Ee
Furthermore, observe that
\Bes
\begin{split}
\left|\langle x, N(x)\rangle_H\right| & \le \int_{\T} |x(\xi)|^2 d\xi+\int_\T |x(\xi)|^4 d\xi \\
& \le \|x\|_H^2+\|x\|^2_\infty \|x\|_H^2  \le \|x\|_H^2+\tl C \|x\|^2_V \|x\|_H^2,
\end{split}
\Ees
where the last inequality is by Sobolev embedding.
Note that for all $s \in (0,T]$
\Bes
\begin{split}
\left|\frac{\langle X^m_s, N(X^m_s)\rangle_H}{(M+\|X^m_s\|^2_H)^{\frac 12}} \right| & \le
\frac{\left(\sup_{0 \le t \le T} \|X^m_t\|^2_H\right)(1+\tl C\|X^m_s\|^2_V)}{(M+\|X^m_s\|_H^2)^{\frac 12}}  \\
& \le  \frac{\left(\sup_{0 \le t \le T} \|X^m_t\|^2_H\right)(1+\tl C\|X_s\|^2_V)}{(M+\|X_s\|_H^2 )^{\frac 12}}, \ \ \ \ \ a.s.
\end{split}
\Ees
Hence,  taking a subsequence if necessary, by \eqref{e:UniXm},\eqref{e:UniXmX} and \eqref{e:IntXsV2} with the Lesbegue dominated convergence theorem,  we obtain
\Be \label{e:Im12Con}
\lim_{m \rightarrow \infty} \int^t_0\frac{\langle X^m_s, N(X^m_s)\rangle_H}{(M+\|X^m_s\|^2_H)^{\frac 12}} \dif s=\int^t_0\frac{\langle X_s, N(X_s)\rangle_H}{(M+\|X^m_s\|^2_H)^{\frac 12}} \dif s, \ \ \ a.s.
\Ee
Combining \eqref{e:Im11Con} and \eqref{e:Im12Con}, we immediately get the desired  equation \eqref{e:I1mCon}.
\end{proof}
\vskip0.3cm

\noindent{\bf Acknowledgments}: The authors are grateful to the anonymous referees for constructive comments
and corrections.  R. Wang thanks the  Faculty of Science and Technology, University of Macau, for finance support and hospitality.    He was supported by Natural Science Foundation of China 11301498, 11431014 and the Fundamental Research Funds for the Central Universities WK0010000048. J. Xiong was  supported by Macao Science and Technology Fund FDCT 076/2012/A3 and Multi-Year Research Grants of the University of Macau Nos. MYRG2014-00015-FST and MYRG2014-00034-FST. L. Xu is supported by the grants: SRG2013-00064-FST, MYRG2015-00021-FST and Science and Technology Development Fund, Macao S.A.R FDCT 049/2014/A1. All of the three authors are supported by the research project RDAO/RTO/0386-SF/2014.

\bibliographystyle{amsplain}

\end{document}